\documentclass[11pt,reqno]{amsart}

\usepackage{yufei}
\title{Log-Concavity in Powers of Infinite Series Close to $(1 - z)^{-1}$}
\date{}
\author{Shengtong Zhang}
\address{Department of Mathematics, Massachusetts Institute of Technology, Cambridge, MA 02139}
\email{stzh1555@mit.edu}
\thanks{All data generated or analysed during this study are included in this published article.}
\begin{document}
\begin{abstract}
In this paper, we use the analytic method of Odlyzko and Richmond to study the log-concavity of power series. If $f(z) = \sum_n a_nz^n$ is an infinite series with $a_n \geq 1$ and $a_0 + \cdots + a_n = O(n + 1)$ for all $n$, we prove that a super-polynomially long initial segment of $f^k(z)$ is log-concave. Furthermore, if there exists constants $C > 1$ and $\alpha < 1$ such that $a_0 + \cdots + a_n = C(n + 1) - R_n$ where $0 \leq R_n \leq O((n + 1)^{\alpha})$, we show that an exponentially long initial segment of $f^k(z)$ is log-concave. This resolves a conjecture proposed by Letong Hong and the author, which implies another conjecture of Heim and Neuhauser that the Nekrasov-Okounkov polynomials $Q_n(z)$ are unimodal for sufficiently large $n$.
\end{abstract}
\maketitle
\section{Introduction}
A polynomial or infinite series $p(z)$ is said to be unimodal / log-concave if the sequence of its coefficients is unimodal / log-concave. More explicitly, let $a_k$ denote the coefficient of $z^k$ in $p(z)$. We say $p(z)$ is \textbf{unimodal} if there exists an index $m$ such that $a_0 \leq a_1 \leq \cdots \leq a_m$ and $a_m \geq a_{m + 1} \geq \cdots$. We say $p(z)$ is \textbf{log-concave} if $a_k \geq 0$ for all $k\geq 0$, and $a_k^2 \geq a_{k - 1}a_{k + 1}$ for all $k \geq 1$. 

Many methods have been developed to show the unimodality and log-concavity of various combinatorial polynomials.  We refer to the excellent surveys of Brenti \cite{Brenti} and Stanley \cite{Stanley, Stanley2} for a collection of these methods, which encompass algebra, combinatorics, and geometry. For example, breakthrough works by Adiprasito, Braden, Huh, Katz, Matherne, Proudfoot, Wang, and many others \cite{Huh, Huh2, Huh3} use methods in algebraic geometry to settle the Mason and Heron–Rota–Welsh
conjecture on the log-concavity of the chromatic polynomial of graphs and the characteristic polynomial of matroids. 

This paper settles the unimodality of the Nekrasov-Okounkov polynomials, an important family of polynomials in combinatorics and representation theory. The Nekrasov-Okounkov polynomials are defined by
$$Q_n(z) := \sum_{\abs{\lambda} = n} \prod_{h \in \cH(\lambda)} \left(1 + \frac{z}{h^2}\right)$$
where $\lambda$ runs over all Young tableaux of size $n$, and $\cH(\lambda)$ denotes the multiset of hook lengths associated with $\lambda$. They are central in the groundbreaking Nekrasov-Okounkov identity \cite{NO}
$$\sum_{n = 0}^{\infty} Q_n(z)q^n = \prod_{m = 1}^{\infty} (1 - q^m)^{-z-1}.$$
In \cite{HN}, Heim and Neuhauser posed the conjecture that $Q_n(z)$ are unimodal polynomials. In \cite{HZ}, the author and Letong Hong showed that for sufficiently large $n$, the unimodality of $Q_n(z)$ is implied by the following conjecture. 
\begin{conjecture}
\label{conj:main}
Let $f(x)$ be the infinity series defined by
\begin{equation}
\label{eq:f-def}
f(z) := \sum_{n \geq 1} \sigma_{-1}(n)z^n
\end{equation}
where $\sigma_{-1}(n) = \sum_{d | n} d^{-1}$. Let $a_{n,k}$ be the coefficient of $z^n$ in $f^k(z)$. There exists a constant $C > 1$ such that for all $k\geq 2$ and $n \leq C^k$, we have
$$a_{n,k}^2 \geq a_{n - 1, k}a_{n + 1, k}.$$
\end{conjecture}

In other words, an exponentially long initial segment of $f^k(z)$ is log-concave as $k$ goes to infinity. See \cite{HN2} for more comments on this conjecture.

This conjecture is closely related to a theorem of Odlyzko and Richmond in 1985. In \cite{OR}, they showed the following remarkable result: if $p(z)$ is a degree $d$ polynomial with non-negative coefficients such that the coefficients $a_0, a_1, a_{d - 1}, a_d$ are all strictly positive, then there exists an $n_0$ such that $p^n(z)$ is log-concave for all $n > n_0$. However, Odlyzko and Richmond pointed out that this result does not trivially generalize to infinite series. 

In this paper, we prove \cref{conj:main}. Our main theorem generalizes \cref{conj:main} to all infinite series $f(z)$ that ``behave like $(1 - z)^{-1}$".

We call an infinite series $f(z)$ \textbf{$1$-lower bounded} if for any $n \geq 0$, the coefficient of $z^n$ in $f(z)$ is at least $1$.
\begin{theorem}
\label{thm:main}
Let 
$$f(z) = \sum_{n = 0}^{\infty} a_n z^n$$
be a $1$-lower bounded infinite series. Suppose there exist constants $c,C > 0$ such that for any $r \in (0,1)$ and $i \in \{0,1,2,3\}$, we have
$$\frac{c}{(1 - r)^{i + 1}} \leq f^{(i)}(r) \leq \frac{C}{(1 - r)^{i + 1}}.$$
Let $a_{n,k}$ denote the coefficient of $z^n$ in the infinite series $f^k(z)$.
Then there exists a constant $A > 1$ depending on $c$ and $C$ only such that 
$$a_{n,k}^2 \geq a_{n - 1, k}a_{n + 1, k}$$ 
for all $n \leq A^{k^{1/3}}$.
\end{theorem}
To rephrase the conditions in the main theorem, we have
\begin{corollary}
\label{cor:handy}
Let 
$$f(z) = \sum_{n = 0}^{\infty} a_n z^n$$
be a $1$-lower bounded infinite series. Suppose there exists a constant $C > 0$ such that for any $n$, we have
$$\frac{1}{n + 1}(a_0 + \cdots + a_n) \leq C$$
Let $a_{n,k}$ denote the coefficient of $z^n$ in the infinite series $f^k(z)$.
Then there exists a constant $A > 1$ depending only on $C$ such that 
$$a_{n,k}^2 \geq a_{n - 1, k}a_{n + 1, k}$$ 
for all $n \leq A^{k^{1/3}}$.
\end{corollary}
We use an ad-hoc argument to strengthen \cref{thm:main} for infinite series that satisfy a stronger condition.
\begin{theorem}
\label{thm:main2}
Let 
$$f(z) = \sum_{n = 0}^{\infty} a_n z^n$$
be a $1$-lower bounded infinite series. Suppose there exists constants $C > 1$, $D > 0$ and $\alpha \in [0,1)$ such that for all $n$, we have
$$0 \leq C(n + 1) - (a_0 + \cdots + a_n) \leq D(n + 1)^{\alpha}.$$
Let $a_{n,k}$ be the coefficient of $z^n$ in the infinite series $f^k(z)$.
Then for sufficiently large $k$ and any $n \leq A^k / k^2$, we have $a_{n,k}^2 \geq a_{n - 1, k}a_{n + 1, k}$. Here $A$ is the constant
$$A := \sqrt[2 + \alpha]{\frac{C}{C - 1}}.$$
\end{theorem}
As a corollary, we establish \cref{conj:main}.
\begin{corollary}
\label{cor:conj-main}
Let $f, a$ be as defined in \cref{conj:main}. Then for any fixed $\eta < \eta_0$, for all sufficiently large $k$ and $n \leq \eta^k$, we have
$$a_{n,k}^2 \geq a_{n - 1, k}a_{n + 1, k}.$$
Here $\eta_0$ is the constant
$$\eta_0 := \sqrt{\frac{\pi^2}{\pi^2 - 6}} > 1.59.$$
\end{corollary}
Combining with Theorem 1.3 of \cite{HZ}, we have shown that
\begin{theorem}
The Nekrasov-Okounkov polynomial $Q_n(z)$ is unimodal for all sufficiently large $n$.
\end{theorem}
\section*{Acknowledgement}
The author thanks Prof. Ken Ono and Letong Hong for their long support and helpful advice throughout this project. The author thanks the anonymous referees for their detailed suggestions.

\section{Proof of \cref{thm:main} and \cref{cor:handy}}
In this section, we show \cref{thm:main}. We use essentially the same method as Odlyzko-Richmond \cite{OR}, with some modifications to accommodate the different behavior of our infinite series $f(z)$. We also note that our notation of $k, n$ is reversed from the convention in \cite{OR}.

 Throughout the proof, for every positive integer $i$, let $c_i$ denote a small positive constant that depends on $c, C$ only, and let $C_i$ denote a large positive constant that depends on $c, C$ only. We denote $c_0 = c$ and $C_0 = C$. To ensure acyclic constant choice, any $c_i, C_i$ is determined by $c_j$ and $C_j$ for $j < i$. We also assume $k$ is sufficiently large relative to all the constants. If $n \leq k^{1/4}$, the proof in \cite{OR} can be applied verbatim to prove \cref{thm:main}. We now prove \cref{thm:main} when $n \geq k^{1/4}$.

The function $f(z)$ is holomorphic for $|z| < 1$. For any $r\in (0,1)$, we have
\begin{align*}
 a_{n,k} &= \frac{1}{2\pi i}\int_{|z| = r}f^k(z)z^{-n-1}dz   \\
 &= \frac{1}{2\pi}r^{-n} \int_{-\pi}^\pi f^k(re^{i\theta})e^{-in\theta}d\theta .
\end{align*}
So it suffices to show that for any $n \in [k^{1/4}, A^{k^{1/3}}]$, there exists an $r \in (0,1)$ such that for $\alpha \in [n - 1, n + 1]$, the real-valued function
$$F(\alpha) = \int_{-\pi}^\pi f^k(re^{i\theta})e^{-i\alpha\theta}d\theta$$
is positive and log-concave. We will show the stronger conclusion that $F$ is concave, that is
$$F''(\alpha) < 0$$
or
$$\int_{-\pi}^\pi \theta^2 f^k(re^{i\theta})e^{-i\alpha\theta}d\theta > 0.$$
The crucial idea is to study the argument of $f(re^{i\theta})$. We first note that $f$ is nonzero in a neighborhood of $z = 1 - r$.
\begin{lemma}
\label{lem:f-lower-bound}
For any $r \in (0,1)$ and $\theta$, we have
$$\abs{f(re^{i\theta})} \geq \left(1 - \frac{C_1\abs{\theta}}{1 - r}\right)f(r).$$ 
In particular, if $\abs{\theta} < c_1(1 - r)$, then $f(re^{i\theta}) \neq 0$.
\end{lemma}
\begin{proof}
By assumption, we have $f(r) \geq c(1 - r)$ and $f'(r) \leq \frac{C}{(1 - r)^2}$. Therefore, $\abs{f'(z)} \leq \frac{C_1}{(1 - \abs{z})}f(r)$ if $\abs{z} \leq r$. By the intermediate value theorem, we conclude that
$$\abs{f(re^{i\theta})} \geq f(r) - \abs{r - re^{i\theta}}\cdot \frac{C_1}{(1 - r)}f(r) \geq \left(1 - \frac{C_1\abs{\theta}}{1 - r}\right)f(r)$$
as desired.
\end{proof}
Therefore, we can define a smooth function $\psi_r(\theta)$ with domain $(-c_1(1 - r), c_1(1 - r))$ such that
$$\psi_r(0) = 0, \psi_r(\theta) \in \arg(f(re^{i\theta})).$$
As $f(\Bar{z}) = \overline{f(z)}$, $\psi_r(\theta)$ is an odd function. A crucial component of the proof is the following estimate of $\psi_r(\theta)$.
\begin{lemma}
\label{lem:arg-estimate}
If $\abs{\theta} \leq c_2(1 - r)$, then we have
$$\abs{\psi_r(\theta) - A(r)\theta} \leq C_2\frac{r\abs{\theta}^3}{(1 - r)^{3}}$$
where $A(r) = \frac{rf'(r)}{f(r)}.$
\end{lemma}
\begin{proof}
We can write
$$\psi_r(\theta) = \Im \ln(f(re^{i\theta})).$$
So we have
$$\psi'_r(\theta) = \Im rie^{i\theta} \frac{f'(re^{i\theta})}{f(re^{i\theta})}.$$
In particular,
$$\psi'_r(0) = r\frac{f'(r)}{f(r)} = A(r).$$
Furthermore, $\psi$ is odd. By Taylor's formula, there exists some $\xi \in (0, \theta)$ such that
$$\psi_r(\theta) = A(r)\theta + \frac{1}{6}\psi_r^{(3)}(\xi)\theta^3.$$
We compute that
\begin{align*}
\psi_r^{(3)}(\xi) &= \Im \left(-ir^3 e^{3i\xi} \frac{f^{(3)}f^3 - 4f'f''f^2 + 2f(f')^3}{f^4}(re^{i\xi}) \right)   \\
&+ \Im\left(- 3ir^2 e^{2i\xi} \frac{f''f - (f')^2}{f^2}(re^{i\xi}) - ire^{i\xi} \frac{f'}{f}(re^{i\xi})\right).
\end{align*}
By \cref{lem:f-lower-bound}, we have $\abs{f(re^{i\xi})} > c(1 - r) / 2$, and by assumption,
$$\abs{f^{(j)}(re^{i\xi})} \leq \abs{f^{(j)}(r)} \leq \frac{C}{(1 - r)^{j + 1}}, j \in \{0,1,2,3\}.$$
We conclude that
$$\abs{\psi_r^{(3)}(\xi)} \leq \frac{C_2r}{(1 - r)^3}.$$
Substituting in, we get the desired estimate.
\end{proof}
As $A(0) = 0$ and $\lim_{r \to 1} A(r) = \infty$, there exists an $r_0 \in (0,1)$ such that $A(r_0) = \frac{n}{k}$. We have $r_0 = \Theta(\min(1, n/k))$. As $n \geq k^{1/4}$, for $k$ sufficiently large we have
$$k r_0 \geq 1.$$
For any $\alpha \in [n - 1, n + 1]$, we now show that
$$\int_{-\pi}^\pi \theta^2 f^k(r_0e^{i\theta})e^{-i\alpha\theta}d\theta > 0.$$
We take
$$\theta_0 = c_3(1 - r_0)(kr_0)^{-1/3}$$
for a $c_3 < c_2$ to be determined later, and split the integral
$$\int_{-\pi}^\pi \theta^2 f^k(r_0e^{i\theta})e^{-i\alpha\theta}d\theta = \int_{\abs{\theta} < \theta_0} \theta^2 f^k(r_0e^{i\theta})e^{-i\alpha\theta}d\theta + \int_{\abs{\theta} \in (\theta_0, \pi)} \theta^2 f^k(r_0e^{i\theta})e^{-i\alpha\theta}d\theta.$$
We now estimate the two summands. By \cref{lem:arg-estimate}, for any $\theta$ with $\abs{\theta} < \theta_0$ we have 
\begin{align*}
\abs{\arg{f^k(r_0e^{i\theta})e^{-i\alpha \theta}}} &= \abs{k\psi_{r_0}(\theta) - \alpha \theta} \\
&\leq k\abs{\psi_{r_0}(\theta) - A(r_0)\theta} + \abs{(n - \alpha)\theta} \\
&\leq \frac{C_2r_0k}{(1 - r_0)^3} \abs{\theta}^3 + \abs{\theta}\leq 2c_3C_2.   
\end{align*}
In particular, if we take the $c_3$ in the definition of $\theta_0$ to be $\min(c_2, \pi / (8C_2))$, then
$$\abs{\arg{f^k(r_0e^{i\theta})e^{-i\alpha \theta}}} < \frac{\pi}{4}.$$
Thus we conclude that
$$\int_{\abs{\theta} < \theta_0} \theta^2 f^k(r_0e^{i\theta})e^{-i\alpha\theta}d\theta \geq \frac{1}{2}\int_{\abs{\theta} < \theta_0} \theta^2 \abs{f^k(r_0e^{i\theta})e^{-i\alpha\theta}}d\theta = \frac{1}{2}\int_{\abs{\theta} < \theta_0} \theta^2 \abs{f(r_0e^{i\theta})}^kd\theta.$$
We apply \cref{lem:f-lower-bound} to obtain
$$\int_{\abs{\theta} < \theta_0} \theta^2 \abs{f(r_0e^{i\theta})}^kd\theta \geq \int_{\abs{\theta} < \theta_0} \theta^2 \left(1 - C_1\frac{\abs{\theta}}{1 - r_0}\right)^k f^k(r_0)d\theta.$$
We can extract the constant $f^k(r_0)$ and apply a change of variable $t = \theta / (1 - r_0)$
$$\int_{\abs{\theta} < \theta_0} \theta^2 \left(1 - C_1\frac{\abs{\theta}}{1 - r_0}\right)^k d \theta = 2(1 - r_0)^3\int_{0}^{\theta_0 / (1 - r_0)} t^2(1 - t)^k d t.$$
As $\theta_0 / (1 - r_0) = c_3(kr_0)^{-1/3} > 2k^{-1}$, we conclude that
$$\int_{0}^{\theta_0 / (1 - r)} t^3(1 - t)^k d t \geq \int_{k^{-1}}^{2k^{-1}} t^3(1 - t)^k dt \geq c_4 k^{-4}.$$
So we obtain the estimate
\begin{equation*}
 \int_{\abs{\theta} < \theta_0} \theta^2 f^k(r_0e^{i\theta})e^{-i\alpha\theta}d\theta \geq c_4(1 - r_0)^3k^{-4}f^k(r_0).   
\end{equation*}
By assumption we have
$$A(r) = \frac{rf'(r)}{f(r)} \geq \frac{rc(1 - r)^{-2}}{C(1 - r)^{-1}} = \frac{c_5r}{1 - r}$$
and as $A(r_0) = n/k$, we get
$$1 - r_0 \geq \frac{c_6k}{\max(n, k)}.$$
So we conclude that
\begin{equation}
\label{eq:major-arc}
\int_{\abs{\theta} < \theta_0} \theta^2 f^k(r_0e^{i\theta})e^{-i\alpha\theta}d\theta \geq c_7\max(n, k)^{-4}f^k(r_0).
\end{equation}
To estimate the second summand, which is the integral over $\abs{\theta} \in (\theta_0, \pi)$, we need a lemma about the upper bound of $f$ away from the positive real axis.
\begin{lemma}
\label{lem:f-upper-bound}
For any $r \in (0,1)$ and $\theta \in [-\pi, \pi)$, we have
$$\abs{f(re^{i\theta})} \leq \left(1 - c_8r\frac{\min(\abs{\theta}, 1 - r)^2}{(1 - r)^2}\right)f(r).$$ 
\end{lemma}
\begin{proof}
As we assumed that $a_n \geq 1$ for every $n$, we have
$$f(z) = \frac{1}{1 - z} + \sum_{n = 0}^{\infty} (a_n - 1)z^n.$$
In particular, we get
$$f(r) = \frac{1}{1 - r} + \sum_{n = 0}^{\infty} (a_n - 1)r^n,$$
$$\abs{f(re^{i\theta})} \leq \abs{\frac{1}{1 - re^{i\theta}}} +  \sum_{n = 0}^{\infty} (a_n - 1)r^n.$$
Thus we have
$$f(r) - \abs{f(re^{i\theta})} \geq \frac{1}{1 - r} - \abs{\frac{1}{1 - re^{i\theta}}} \geq \frac{c_8r\min(\abs{\theta}, 1 - r)^2}{(1 - r)^3}$$
where the second inequality follows from
$$\frac{1}{1 - r} - \abs{\frac{1}{1 - re^{i\theta}}} = \frac{\sqrt{(1 - r)^2 + 4\sin^2\frac{\theta}{2}} - (1 - r)}{(1 - r)\abs{1 - re^{i\theta}}}.$$
As $f(r) \leq C(1 - r)^{-1}$, we conclude that
$$f(r) - \abs{f(re^{i\theta})} \geq \frac{1}{1 - r} - \abs{\frac{1}{1 - re^{i\theta}}} \geq c_8r\frac{\min(\abs{\theta}, 1 - r)^2}{(1 - r)^2} f(r)$$
as desired.
\end{proof}
By the lemma, for every $\theta$ with $\abs{\theta} > \theta_0 = c_3(1 - r_0)(kr_0)^{-1/3}$, we have
$$\abs{f(r_0e^{i\theta})} \leq \left(1 - \frac{c_8r_0\abs{\theta_0}^2}{(1 - r_0)^2}\right)f(r_0).$$
Thus we conclude that
\begin{align*}
 \abs{\int_{\abs{\theta} \in (\theta_0, \pi)} \theta^2 f^k(r_0e^{i\theta})e^{-i\alpha\theta}d\theta} &\leq \pi^3 \left(1 - \frac{c_8r_0\abs{\theta_0}^2}{(1 - r_0)^2}\right)^kf^k(r_0) \\
 &\leq \pi^3 e^{-c_8 kr_0\frac{cr_0\abs{\theta_0}^2}{(1 - r_0)^2}}f^k(r_0).
\end{align*}
Substituting the value of $\theta_0$, we get
\begin{equation}
\label{eq:minor-arc}
  \abs{\int_{\abs{\theta} \in (\theta_0, \pi)} \theta^2 f^k(r_0e^{i\theta})e^{-i\alpha\theta}d\theta} \leq \pi^3 e^{-c_9(kr_0)^{1/3}}f^k(r_0).  
\end{equation}
Finally, we combine the estimates \eqref{eq:major-arc} and \eqref{eq:minor-arc} to conclude that
$$\int_{-\pi}^\pi \theta^2 f^k(r_0e^{i\theta})e^{-i\alpha\theta}d\theta \geq c_7\max(n, k)^{-4}f^k(r_0) - \pi^3 e^{-c_9(kr_0)^{1/3}}f^k(r_0).$$
We note that
$$\frac{n}{k} = A(r_0) = \frac{r_0f'(r_0)}{f(r_0)} \leq \frac{r_0C(1 - r_0)^{-2}}{c(1 - r_0)^{-1}} = \frac{C_{10}r_0}{(1 - r_0)}$$
which implies
$$r_0 \geq \frac{n}{n + C_{10}k}.$$
If $C_{10}k < n$, then for sufficiently large $k$, we have
$$c_7\max(n, k)^{-4} - \pi^3 e^{-c_9(kr_0)^{1/3}} \geq c_7n^{-4} - C_{11} e^{-c_{10}k^{1/3}}.$$ 
So there exists a $c_{11} > 0$ such that if $C_{10}k < n \leq e^{c_{11}k^{1/3}}$, then
$$c_7\max(n, k)^{-4} - \pi^3 e^{-c_9(kr_0)^{1/3}} > 0.$$
If $k^{1/4} \leq n \leq C_{10}k$, then for sufficiently large $k$, we have
$$c_{7}\max(n, k)^{-4} - \pi^3 e^{-c_9(kr_0)^{1/3}} \geq c_{12}k^{-4} - \pi^3 e^{-c_{10}(k\cdot n/k)^{1/3}} \geq c_{12}k^{-4} - C_{12} e^{-c_{10}k^{1/12}} > 0.$$
In both cases we have
$$\int_{-\pi}^\pi \theta^2 f^k(r_0e^{i\theta})e^{-i\alpha\theta}d\theta > 0.$$
Thus, we have shown \cref{thm:main}.

\cref{cor:handy} is an easy corollary of \cref{thm:main}.
\begin{proof}[Proof of \cref{cor:handy}]
For each non-negative integer $i$ and $r \in (0,1)$, we note that
$$f^{(i)}(r) = i!\sum_{n = 0}^{\infty} \binom{n + i}{i}a_{n + i}r^n.$$
On one hand, we have
$$f^{(i)}(r) \geq i!\sum_{n = 0}^{\infty} \binom{n + i}{i}r^n = \frac{i!}{(1 - r)^{i + 1}}.$$
On the other hand, by Abel summation, we have
\begin{align*}
 f^{(i)}(r) &= i!\sum_{n = 0}^{\infty} (r^n - r^{n + 1}) \sum_{k = 0}^n \binom{k + i}{i}a_{k + i} \\
 &\leq i!\sum_{n = 0}^{\infty}(r^n - r^{n + 1}) \cdot C(n + i + 1) \cdot  \binom{n + i}{i}   \\
 &= C(i + 1)!\sum_{n = 0}^{\infty}\binom{n + i + 1}{i + 1}(r^n - r^{n + 1}) = \frac{C(i + 1)!}{(1 - r)^{i + 1}}.
\end{align*}
Thus for any $i \geq 0$ and $r \in (0,1)$ we have
$$\frac{i!}{(1 - r)^{i + 1}} \leq f^{(i)}(r)\leq \frac{C(i + 1)!}{(1 - r)^{i + 1}}.$$
The corollary follows by applying \cref{thm:main}.
\end{proof}
\section{Proof of \cref{thm:main2} and \cref{cor:conj-main}}
In this section, we assume that $f(z) = \sum_n a_nz^n$ satisfies the condition of \cref{thm:main2}: For all $n$, we have $a_n \geq 1$ and 
\begin{equation}
\label{eq:condition}
0 \leq C(n + 1) - (a_0 + \cdots + a_n) \leq D(n + 1)^{\alpha}.    
\end{equation}
We observe that $f(z)$ also satisfies the condition in \cref{cor:handy}, so there exists a $B > 1$ such that $a_{n,k}^2 \geq a_{n - 1,k}a_{n + 1, k}$ for all $n \leq B^{k^{1/3}}$. We now use a different method to show that $a_{n,k}^2 \geq a_{n - 1,k}a_{n + 1, k}$ for all $B^{k^{1/3}} \leq n \leq A^k / k^2$ and $k$ sufficiently large.

The inequality $a_{n,k}^2 \geq a_{n - 1,k}a_{n + 1, k}$ is equivalent to the inequality
$$(a_{n,k} - a_{n - 1, k})^2 \geq a_{n - 1, k} (a_{n + 1, k} - 2a_{n, k} + a_{n - 1, k}).$$
The key observation is that the second order difference $a_{n + 1, k} - 2a_{n, k} + a_{n - 1, k}$ can be bounded. 

We introduce a notation: for any $n \geq 0$, define $a_n^{(0)} = 1$ and $a_n^{(1)} = a_n - 1$. Then we have
\begin{align*}
    a_{n, k} &= \sum_{x_1 + \cdots + x_k = n} a_{x_1}a_{x_2}\cdots a_{x_k} \\
    &= \sum_{x_1 + \cdots + x_k = n} \prod_{i = 1}^k \left(a^{(0)}_{x_i} + a^{(1)}_{x_i}\right) \\
    &= \sum_{(i_1, i_2,\cdots,i_k) \in \{0,1\}^k} \sum_{x_1 + \cdots + x_k = n}a^{(i_1)}_{x_1}\cdots a^{(i_k)}_{x_k}.
\end{align*}
For a tuple $I = (i_1, i_2,\cdots,i_k) \in \{0,1\}^k$, we let
$$a^I_n = \sum_{x_1 + \cdots + x_k = n}a^{(i_1)}_{x_1}\cdots a^{(i_k)}_{x_k}.$$
Then we have
$$a_{n,k} = \sum_{I \in \{0,1\}^k}a_n^I.$$
Let $\mathbf{1}_{k - 1}$ denote the length $k - 1$ tuple $(1,1,\cdots, 1)$ and $(\mathbf{1}_{k - 1}, 0)$ denote the length $k$ tuple $(1,1,\cdots, 1,0)$.

We prove a series of lemmas that gives the desired control over the second-order difference.
\begin{lemma}
\label{lem:one-zero}
There exists a constant $C_1 > 0$ such that for any $n$ and $k \geq 2$, we have
$$a^{(\mathbf{1}_{k - 1}, 0)}_n = \binom{n + k - 1}{k - 1}(C - 1)^{k - 1}\left(1 - R_{n,k}\right)$$
where 
$$0 \leq R_{n,k} \leq \frac{C_1k(k - 1)}{(n + k - 1)^{1 - \alpha}}.$$
\end{lemma}
\begin{proof}
We take $C_1 = D / \min(C - 1,1)$, and argue by induction on $k$. For $k = 2$, the statement is clear as
$$a^{(\mathbf{1}_1, 0)}_n = a_0 + \cdots + a_n - (n + 1).$$
So \eqref{eq:condition} implies
$$0 \leq R_{n,2} \leq \frac{D}{(n + 1)^{1 - \alpha}}.$$
Now suppose the lemma holds for $k' = k - 1$. To prove the lemma for $k$, we observe
$$a^{(\mathbf{1}_{k - 1}, 0)}_n = \sum_{x_1 + x_2 = n}a^{(1)}_{x_1}a^{(\mathbf{1}_{k - 2}, 0)}_{x_2} = \sum_{x_1 + x_2 = n}(a^{(1)}_0 + \cdots + a^{(1)}_{x_1})(a^{(\mathbf{1}_{k - 2}, 0)}_{x_2} - a^{(\mathbf{1}_{k - 2}, 0)}_{x_2 - 1}).$$
Using \eqref{eq:condition}, we have
$$a^{(\mathbf{1}_{k - 1}, 0)}_n = (C - 1)\sum_{x_1 + x_2 = n}(x_1 + 1)\left(a^{(\mathbf{1}_{k - 2}, 0)}_{x_2} - a^{(\mathbf{1}_{k - 2}, 0)}_{x_2 - 1}\right) - S_{n,k}.$$
where
$$S_{n,k} = \sum_{x_1 + x_2 = n} ((C - 1)(x_1 + 1) - a_0^{(1)} - \cdots - a_{x_1}^{(1)})\left(a^{(\mathbf{1}_{k - 2}, 0)}_{x_2} - a^{(\mathbf{1}_{k - 2}, 0)}_{x_2 - 1}\right).$$
We first continue estimating the main term. By the induction hypothesis, we have
\begin{align*}
&(C - 1)\sum_{x_1 + x_2 = n}(x_1 + 1)\left(a^{(\mathbf{1}_{k - 2}, 0)}_{x_2} - a^{(\mathbf{1}_{k - 2}, 0)}_{x_2 - 1}\right) \\
=& (C - 1)\sum_{x_2 = 0}^n a^{(\mathbf{1}_{k - 2}, 0)}_{x_2} \\    
=& (C - 1)^{k - 1}\left(\sum_{m = 0}^{n} \binom{m + k - 2}{k - 2} - \binom{m + k - 2}{k - 2}R_{m,k - 1}\right) \\
=& (C - 1)^{k - 1}\left(\binom{n + k - 1}{k - 1} - \sum_{m = 0}^{n} \binom{m + k - 2}{k - 2}R_{m,k - 1}\right). 
\end{align*}
By the induction hypothesis, the subtracted term is positive. Again by the induction hypothesis, we estimate that
\begin{align*}
&(C - 1)^{k - 1}\sum_{m = 0}^{n} \binom{m + k - 2}{k - 2}R_{m,k - 1} \\
\leq&(C - 1)^{k - 1} \sum_{m = 0}^{n} \binom{m + k - 2}{k - 2} \frac{C_1(k - 1)(k - 2)}{(m + k - 2)^{1 - \alpha}} \\
=& (C - 1)^{k - 1}\sum_{m = 0}^{n} \binom{m + k - 3}{k - 3} C_1(k - 1)(m + k - 2)^{\alpha} \\
\leq& (C - 1)^{k - 1}\sum_{m = 0}^{n} \binom{m + k - 3}{k - 3} C_1(k - 1)(n + k - 2)^{\alpha} \\
=& (C - 1)^{k - 1}\binom{n + k - 2}{k - 2} C_1(k - 1)(m + k - 2)^{\alpha} \\
\leq& \binom{n + k - 1}{k - 1}(C - 1)^{k - 1}  \frac{C_1(k - 1)^2}{(m + k - 2)^{1 - \alpha}}.
\end{align*}
To estimate error term $S_{n,k}$, we first note that 
$$a^{(\mathbf{1}_{k - 2}, 0)}_{x_2} = \sum_{x = 0}^{x_2} a^{\mathbf{1}_{k - 2}}_{x}$$
so $a^{(\mathbf{1}_{k - 2}, 0)}_{x_2} \geq a^{(\mathbf{1}_{k - 2}, 0)}_{x_2 - 1}$ for any $x_2$. By \eqref{eq:condition}, we conclude that $S_{n,k}$ is non-negative. On the other hand, by \eqref{eq:condition} and the induction hypothesis we have
\begin{align*}
S_{n,k} &\leq \sum_{x_1 + x_2 = n} D(x_1 + 1)^{\alpha} (a^{(\mathbf{1}_{k - 2}, 0)}_{x_2} - a^{(\mathbf{1}_{k - 2}, 0)}_{x_2 - 1}) \\
&= \sum_{x_1 + x_2 = n} D((x_1 + 1)^{\alpha} - x_1^{\alpha}) a^{(\mathbf{1}_{k - 2}, 0)}_{x_2} \\   
&\leq \sum_{x_2 = 0}^n D((n + 1 - x_2)^{\alpha} - (n - x_2)^{\alpha}) \binom{x_2 + k - 2}{k - 2}C^{k - 2}.
\end{align*}
We estimate that
\begin{align*}
S_{n,k} &\leq \sum_{x_2 = 0}^n D((n + 1 - x_2)^{\alpha} - (n - x_2)^{\alpha}) \binom{n + k - 2}{k - 2}(C - 1)^{k - 2} \\
&= D(n + 1)^{\alpha} \binom{n + k - 2}{k - 2}(C - 1)^{k - 2} \\
&\leq D \binom{n + k - 1}{k - 1}(C - 1)^{k - 2} \frac{k - 1}{(n + k - 1)^{1 - \alpha}} \\
&\leq \binom{n + k - 1}{k - 1}(C - 1)^{k - 1} \frac{C_1(k - 1)}{(n + k - 1)^{1 - \alpha}}.
\end{align*} 
Combining all the estimates, we conclude that
$$a_n^{(\mathbf{1}_{k - 1},0)} = \binom{n + k - 1}{k - 1}(C - 1)^{k - 1}\left(1 - R_{n,k}\right)$$
where
$$0 \leq R_{n,k} \leq \frac{C_1(k - 1)^2}{(n + k - 1)^{1 - \alpha}} + \frac{C_1(k - 1)}{(n + k - 1)^{1 - \alpha}} = \frac{C_1k(k - 1)}{(n + k - 1)^{1 - \alpha}}$$
as desired.
\end{proof}
\begin{lemma}
\label{lem:at-least-one-zero}
For any $n$ and $k \geq 2$, if a tuple $I \in \{0,1\}^k$ has $k_0$ zeros and $k_1$ ones with $k_0 \geq 1$, then
$$a^{I}_n = \binom{n + k - 1}{k - 1}(C - 1)^{k_1}\left(1 - S^{(0)}_{n,I}\right)$$
where 
$$0 \leq S^{(0)}_{n,I} \leq \frac{C_1k(k - 1)}{(n + k - 1)^{1 - \alpha}}.$$
\end{lemma}
\begin{proof}
By definition, permuting the entries of $I$ does not change the value of $a^I_{n}$, so without loss of generality we can assume $I = (1,\cdots,1,0,\cdots,0)$. If $k_0 = 1$ then the lemma is precisely \cref{lem:one-zero}, so we assume $k_0 \geq 2$. If $k_1 = 0$ then 
$$a^{I}_n = \sum_{x_1 + \cdots + x_k = n} 1 = \binom{n + k - 1}{k - 1}$$
so $S^{(0)}_{n,I} = 0$, and the lemma is obvious. Now assume $k_1 \geq 1$. We have
\begin{align*}
a^{I}_n &= \sum_{x_1 + x_2 = n} a^{(\mathbf{1}_{k_1}, 0)}_{x_1}a^{(0,\cdots,0)}_{x_2} \\
&= \sum_{x_1 + x_2 = n} a^{(\mathbf{1}_{k_1}, 0)}_{x_1}\binom{x_2 + k_0 - 2}{k_0 - 2} \\
&= \sum_{x_1 + x_2 = n} \binom{x_1 + k_1}{k_1}(C - 1)^{k_1}\left(1 - R_{x_1,k_1 + 1}\right)\binom{x_2 + k_0 - 2}{k_0 - 2} \\
&= \binom{n + k - 1}{k - 1}(C - 1)^{k_1} - \sum_{x_1 + x_2 = n} \binom{x_1 + k_1}{k_1}(C - 1)^{k_1} R_{x_1,k_1 + 1}\binom{x_2 + k_0 - 2}{k_0 - 2}.
\end{align*}
By \cref{lem:one-zero} we have the bound
$$0 \leq R_{x_1,k_1 + 1} \leq \frac{C_1(k_1 + 1)k_1}{(x_1 + k_1)^{1 - \alpha}}.$$
Thus $S^{(0)}_{n,I} \geq 0$. We also have the upper bound
\begin{align*}
    S^{(0)}_{n,I} =& \sum_{x_1 + x_2 = n} \binom{x_1 + k_1}{k_1}(C - 1)^{k_1} R_{x_1,k_1 + 1}\binom{x_2 + k_0 - 2}{k_0 - 2} \\
    \leq& \sum_{x_1 + x_2 = n} \binom{x_1 + k_1}{k_1}(C - 1)^{k_1} \frac{C_1(k_1 + 1)k_1}{(x_1 + k_1)^{1 - \alpha}}\binom{x_2 + k_0 - 2}{k_0 - 2} \\
    \leq& \sum_{x_1 + x_2 = n} \binom{x_1 + k_1 - 1}{k_1 - 1}(C - 1)^{k_1} C_1(k_1 + 1)\cdot (x_1 + k_1)^{\alpha}\binom{x_2 + k_0 - 2}{k_0 - 2} \\
    \leq& (C - 1)^{k_1} C_1(k_1 + 1)\cdot (n + k_1)^{\alpha} \sum_{x_1 + x_2 = n} \binom{x_1 + k_1 - 1}{k_1 - 1}\binom{x_2 + k_0 - 2}{k_0 - 2} \\
    \leq&(C - 1)^{k_1} C_1(k_1 + 1)\cdot (n + k_1)^{\alpha}  \binom{n + k - 2}{k - 2} \\
    \leq&(C - 1)^{k_1}  \binom{n + k - 1}{k - 1} \cdot C_1 \frac{(k - 1)k}{(n + k - 1)^{1 - \alpha}}.
\end{align*}
So we have the desired inequality
$$S^{(0)}_{n, I} \leq \frac{C_1k(k - 1)}{(n + k - 1)^{1 - \alpha}}.$$
\end{proof}
We arrive at the crucial second-order difference estimates.
\begin{lemma}
\label{lem:2nd-order-aux}
There exists a constant $C_2 > 0$ such that for any $n \geq -1$ and $k \geq 3$, if the tuple $I \in \{0,1\}^k$ has $k_0$ zeros and $k_1$ ones with $k_0 \geq 3$, then
$$a^{I}_{n + 1} - 2a^{I}_n + a^{I}_{n - 1} = \frac{(n + k)^{k - 3}}{(k - 3)!}(C - 1)^{k_1}\left(1 - S^{(2)}_{n,I}\right)$$
where 
$$0 \leq S^{(2)}_{n,I} \leq \frac{C_2k^2}{(n + k)^{1 - \alpha}}.$$
\end{lemma}
\begin{proof}
If $I'$ is the tuple obtained by removing two zeros from $I$, then
$$a^{I}_n = \sum_{x_1 + x_2 = n} a^{I'}_{x_1}a^{(0,0)}_{x_2} = \sum_{x_1 + x_2 = n} a^{I'}_{x_1} (x_2 + 1)\mathbf{1}_{x_2 \geq 0}.$$
Thus we find that
\begin{align*}
a^{I}_{n + 1} - 2a^{I}_n + a^{I}_{n - 1} &=  \sum_{x_1 = 0}^{n + 1} a^{I'}_{x_1} ((n - x_1 + 2)\mathbf{1}_{x_1 \leq n + 1} - 2(n - x_1 + 1)\mathbf{1}_{x_1 \leq n} + (n - x_1)\mathbf{1}_{x_1 \leq n - 1}) \\    
&= a^{I'}_{n + 1}.
\end{align*}
Applying \cref{lem:at-least-one-zero}, we obtain
$$a^{I'}_{n + 1} = \binom{n + k - 2}{k - 3}(C - 1)^{k_1}\left(1 - S^{(0)}_{n + 1,I'}\right)$$
where
$$0 \leq S^{(0)}_{n + 1,I'} \leq \frac{C_1(k - 2)(k - 3)}{(n + k - 2)^{1 - \alpha}} \leq \frac{3C_1k^2}{(n + k)^{1 - \alpha}}.$$
Finally, we note that
$$\binom{n + k - 2}{k - 3} = \frac{(n + k)^{k - 3}}{(k - 3)!} (1 - S_{n, I}^{(1)})$$
where
$$0 \leq S_{n, I}^{(1)} = 1 - \prod_{i = 2}^{k - 2} \left(1 - \frac{i}{n + k}\right) \leq \frac{k^2}{n + k}.$$
The error term $S_{n,I}^{(2)}$ satisfies
$$1 - S_{n,I}^{(2)} = (1 - S^{(0)}_{n + 1,I'})(1 - S_{n,I}^{(1)}).$$
Therefore we have
$$0 \leq S_{n,I}^{(2)} \leq S^{(0)}_{n + 1,I'}+S_{n,I}^{(1)} $$
and the desired estimate follows.
\end{proof}
\begin{lemma}
\label{lem:2nd-order}
For any $n \geq -1,k \geq 3$, we have
$$a_{n + 1, k} - 2a_{n, k} + a_{n - 1, k} = C^{k}\frac{(n + k)^{k - 3}}{(k - 3)!}(1 + R^{(2)}_{n,k}).$$
where $R^{(2)}_{n, k}$ satisfies
$$\abs{R^{(2)}_{n,k}} \leq E\left(\frac{k^2}{(n + k)^{1 - \alpha}} + (n + k)^{2 + \alpha}A^{-(2 + \alpha)k}\right)$$
for some constant $E > 0$.
\end{lemma}
\begin{proof}
Throughout the proof, we use $R_i$ to denote the various error term that contribute to $R^{(2)}_{n,k}$. Recall the identity
\begin{align*}
 a_{n, k} &= \sum_{I \in \{0,1\}^k} a^I_{n}
\end{align*}
We split the sum into two parts. Let $S_1$ be the set of $I \in \{0,1\}^k$ with at least three ones, and let $S_2$ be the set of $I \in \{0,1\}^k$ with at most $2$ ones. Then
\begin{align*}
a_{n, k} =& \sum_{I \in S_1} a^I_{n} + \sum_{I \in S_2} a^I_{n}
\end{align*}
Let $k_1(I)$ denote the number of ones in $I$. By \cref{lem:2nd-order-aux}, the second-order difference of the first term is
\begin{equation}
\label{eq:main-term-2nd}
\sum_{I \in S_1}\frac{(n + k)^{k - 3}}{(k - 3)!}(C - 1)^{k_1(I)}+ R_1    
\end{equation}
where 
\begin{equation*}
\abs{R_1} \leq \sum_{I \in S_1} \frac{(n + k)^{k - 3}}{(k - 3)!}(C - 1)^{k_1(I)}\abs{S_{n,I}^{(2)}}\leq \sum_{I \in S_1}\frac{(n + k)^{k - 3}}{(k - 3)!}(C - 1)^{k_1(I)} \cdot \frac{C_2k^2}{(n + k)^{1 - \alpha}}.
\end{equation*}
The residue $R_1$ of \eqref{eq:main-term-2nd} is bounded by
\begin{equation}
    \label{eq:R1}
    \abs{R_1} \leq \sum_{I \in \{0,1\}^k}\frac{(n + k)^{k - 3}}{(k - 3)!}(C - 1)^{k_1(I)} \cdot \frac{C_2k^2}{(n + k)^{1 - \alpha}} = \frac{(n + k)^{k - 3}}{(k - 3)!}C^k \cdot \frac{C_2k^2}{(n + k)^{1 - \alpha}}.
\end{equation}
The main term of \eqref{eq:main-term-2nd} satisfies
\begin{align*}
&\sum_{I \in S_1}\frac{(n + k)^{k - 3}}{(k - 3)!}(C - 1)^{k_1(I)} \\
=& \frac{(n + k)^{k - 3}}{(k - 3)!}\sum_{I \in \{0,1\}^k}(C - 1)^{k_1(I)}
- \frac{(n + k)^{k - 3}}{(k - 3)!}\sum_{I \in S_2}(C - 1)^{k_1(I)} \\
=& \frac{(n + k)^{k - 3}}{(k - 3)!}C^{k} - \frac{(n + k)^{k - 3}}{(k - 3)!}\sum_{I \in S_2}(C - 1)^{k_1(I)},
\end{align*}
Let $R_2$ denote
$$R_2 := -\frac{(n + k)^{k - 3}}{(k - 3)!}\sum_{I \in S_2}(C - 1)^{k_1(I)}.$$
Note that
\begin{equation*}
\sum_{I \in S_2}(C - 1)^{k_1(I)} = \binom{k}{2}(C - 1)^{k - 2} + k(C - 1)^{k - 1} + (C - 1)^k \leq k^2C^2(C - 1)^{k - 2}
\end{equation*}
so
\begin{equation}
\label{eq:R2}
\abs{R_2} \leq \frac{(n + k)^{k - 3}}{(k - 3)!} k^2C^2(C - 1)^{k - 2}.
\end{equation}
Thus we conclude that
$$\sum_{I \in S_1} a^I_{n + 1} - 2\sum_{I \in S_1} a^I_{n} + \sum_{I \in S_1} a^I_{n - 1} = \frac{(n + k)^{k - 3}}{(k - 3)!}C^{k} + R_1 + R_2$$
where $R_1, R_2$ are controlled by \eqref{eq:R1} and \eqref{eq:R2} respectively.

It remains to estimate
$$\sum_{I \in S_2} a^I_{n}.$$
For each $I = (i_1,\cdots,i_n)\in S_2$, we have
$$a^I_n = \sum_{x_1 + \cdots + x_k = n} a^{(i_1)}_{x_1}\cdots a^{(i_n)}_{x_n}.$$
By \cref{eq:condition}, we have $a_n \leq C + Dn^{\alpha} \leq (C + D)n^{\alpha}$. Thus
$$\sum_{x_1 + \cdots + x_k = n} a^{(i_1)}_{x_1}\cdots a^{(i_k)}_{x_k} \leq (C + D)n^{\alpha}\sum_{x_1 + \cdots + x_k = n} a^{(i_2)}_{x_2}\cdots a^{(i_k)}_{x_k}.$$
We can appeal to \cref{lem:at-least-one-zero} to obtain
$$\sum_{x_1 + \cdots + x_k = n} a^{(i_2)}_{x_2}\cdots a^{(i_k)}_{x_k} \leq \binom{n + k - 1}{k - 1}(C - 1)^{k_1((i_2,\cdots,i_n))}\leq \binom{n + k - 1}{k - 1}(C - 1)^{k_1(I) - 1}C.$$
Thus we obtain
\begin{align*}
\sum_{I \in S_2} a^I_{n} &\leq (C + D)n^{\alpha} \sum_{I \in S_2} \binom{n + k - 1}{k - 1}(C - 1)^{k_1(I) - 1}C \\
&\leq (C + D)n^{\alpha} \binom{n + k - 1}{k - 1}\cdot k(k - 1)C^3(C - 1)^{k - 3}.    
\end{align*}
We conclude that
\begin{equation*}
\sum_{I \in S_2} a^I_{n} \leq 3(C + D)C^3(C - 1)^{k - 3} (n + k)^{2 + \alpha} \frac{(n + k)^{k - 3}}{(k - 3)!}
\end{equation*}
Thus the second order difference $R_3$ of $\sum_{I \in S_2} a^I_{n}$ is bounded by
\begin{equation}
\label{eq:R3}
R_3 \leq E_1 n^{2 + \alpha} \frac{(n + k)^{k - 3}}{(k - 3)!} (C - 1)^{k}
\end{equation}
for some constant $E_1$.

We have thus finished the second order difference estimate
$$a_{n + 1, k} - 2a_{n,k} + a_{n - 1, k} = \frac{(n + k)^{k - 3}}{(k - 3)!}C^k + R_1 + R_2 + R_3$$
where the errors $R_i$ satisfy \eqref{eq:R1}, \eqref{eq:R2} and \eqref{eq:R3} respectively. We now note that the absolute value of each $R_i$ is at most a constant times
$$\frac{(n + k)^{k - 3}}{(k - 3)!}C^k \cdot \left(\frac{k^2}{(n + k)^{1 - \alpha}} + (n + k)^{2 + \alpha} \cdot A^{-(2 + \alpha)k}\right).$$
Thus we obtain the desired estimate.
\end{proof}
Using the identity
$$a_{n,k} - a_{n - 1, k} = \sum_{n' = -1}^{n - 1} (a_{n' + 1,k} - 2a_{n', k} + a_{n' - 1, k})$$
We conclude an analogous estimate on the first-order difference.
\begin{corollary}
\label{cor:1st-order}
For any $n\geq 0$ and $k \geq 3$, we have
$$a_{n, k} - a_{n - 1, k} = C^{k}\frac{(n + k)^{k - 2}}{(k - 2)!}(1 + R^{(1)}_{n,k}).$$
where $R^{(1)}_{n, k}$ satisfies
$$\abs{R^{(1)}_{n,k}} \leq E\left(\frac{k^2}{(n + k)^{1 - \alpha}} + (n + k)^{2 + \alpha}A^{-(2 + \alpha)k}\right)$$
for some constant $E$.
\end{corollary}
Again using the identity
$$a_{n-1,k} = \sum_{n' = 0}^{n-1} (a_{n',k} - a_{n' - 1, k})$$
We conclude an analogous estimate on the zeroth-order difference.
\begin{corollary}
\label{cor:0nd-order}
For any $n\geq 1$ and $k \geq 3$, we have
$$a_{n-1, k}= C^{k}\frac{(n + k)^{k - 1}}{(k - 1)!}(1 + R^{(0)}_{n,k}).$$
where $R^{(0)}_{n, k}$ satisfies
$$\abs{R^{(0)}_{n,k}} \leq E\left(\frac{k^2}{(n + k)^{1 - \alpha}} + (n + k)^{2 + \alpha}A^{-(2 + \alpha)k}\right)$$
for some constant $E$.
\end{corollary}
Finally, we conclude by \cref{lem:2nd-order}, \cref{cor:1st-order} and \cref{cor:0nd-order} that for any $n \geq -1$, we have
$$\frac{a_{n - 1, k} (a_{n + 1, k} - 2a_{n, k} + a_{n - 1, k})}{(a_{n,k} - a_{n - 1, k})^2} = \frac{k - 2}{k - 1} \cdot  \frac{(1 + R_{n,k}^{(0)})(1 + R_{n,k}^{(2)})}{(1 + R_{n,k}^{(1)})^2}$$
where for each $i \in \{0,1,2\}$ we have
$$\abs{R_{n,k}^{(i)}} \leq E_i\left(\frac{k^2}{(n + k)^{1 - \alpha}} + (n + k)^{2 + \alpha}A^{-(2 + \alpha)k}\right)$$
for constants $E_0, E_1, E_2$. If
$$k^{5/(1 - \alpha)} \leq n \leq \frac{A^k}{k^2}$$
then for sufficiently large $k$, we have
$$\abs{R_{n,k}^{(i)}} \leq \frac{1}{k^2}$$
for each $i \in \{0,1,2\}$. Therefore we get
$$\frac{a_{n - 1, k} (a_{n + 1, k} - 2a_{n, k} + a_{n - 1, k})}{(a_{n,k} - a_{n - 1, k})^2} \leq \frac{k - 2}{k - 1} \cdot  \left(1 + \frac{1}{k^2}\right)^4 < 1.$$
So $\{a_{n, k}\}$ is log-concave for $k^{5/(1 - \alpha)} \leq n \leq \frac{\eta^k}{k^2}$. As we have shown that $\{a_{n, k}\}$ is log-concave for $n \leq B^{k^{1/3}}$, where $B > 1$ is a constant, the two intervals glue together to obtain \cref{thm:main2}.

Finally, we prove \cref{cor:conj-main}. Let $f(z)$ be defined in \cref{conj:main} and let 
$$g(z) := \frac{f(z)}{z} = \sum_{n = 0}^{\infty} \sigma_{-1}(n + 1)z^n.$$  As $\sigma_{-1}(n) \geq 1$ for any $n \geq 1$, $g$ is $1$-lower bounded. Furthermore, we have
$$\sigma_{-1}(1) + \cdots + \sigma_{-1}(n + 1) = \sum_{m = 1}^{n + 1}\sum_{d | m} \frac{1}{d} = \sum_{d = 1}^{n + 1} \frac{1}{d}\floor{\frac{n + 1}{d}}.$$
Thus we have
$$\sigma_{-1}(1) + \cdots + \sigma_{-1}(n + 1) \leq \sum_{d = 1}^{\infty}\frac{n + 1}{d^2} = \frac{\pi^2}{6}(n + 1)$$
and
$$\sigma_{-1}(1) + \cdots + \sigma_{-1}(n + 1) \geq \sum_{d = 1}^{n + 1} \left(\frac{n + 1}{d^2} - \frac{d - 1}{d^2}\right) \geq \frac{\pi^2}{6}(n + 1) - \log(n + 1) - 1.$$
So $g(z)$ satisfies the condition of \cref{thm:main2} for $C = \pi^2 / 6$ and any $\alpha > 0$. \cref{cor:conj-main} then follows from \cref{thm:main2}.

\end{document}